\documentclass[12pt,a4paper]{amsart}
\usepackage[utf8]{inputenc}
\usepackage{amsmath}
\usepackage{amsfonts}
\usepackage{amssymb}
\usepackage{graphicx}
\usepackage{amsmath,amssymb, amsthm}
\usepackage{pgf}
\usepackage{tikz-cd}
\usepackage{xy}
\usepackage{hyperref}
\usetikzlibrary{matrix,arrows,decorations.pathmorphing}
\usepackage[english]{babel}
\usepackage{mathtools}
\usepackage{caption}
\usepackage{geometry}
\usepackage{float}

\theoremstyle{definition}

\theoremstyle{definition}

\theoremstyle{definition}

\theoremstyle{definition}\usepackage{amsmath}

\theoremstyle{plain}
\newtheorem{thm}{Theorem}
\theoremstyle{plain}
\newtheorem{prop}{Proposition}
\theoremstyle{plain}

\theoremstyle{plain}
\newtheorem{lem}{Lemma}

\newtheorem{coro}{Corollary}

\newcommand{\Z}{\mathbb{Z}}
\newcommand{\Q}{\mathbb{Q}}

\newcommand{\C}{\mathbb{C}}

\newcommand{\Ok}{\mathcal{O}_K}

\newcommand{\Tr}{\operatorname{Tr}}

\newcommand{\N}{\mathbb{N}}
\newcommand{\pp}{\mathfrak{p}}

\newcommand{\Dk}{\mathcal{D}_K}

\address{Dipartimento di Matematica\\
         Università di Milano\\
         via Saldini 50\\
         20133 Milano\\
         Italy}
\email{francesco.battistoni@unimi.it}

\title[Discriminants of number fields and surjectivity of trace homomorphism]{Discriminants of number fields and surjectivity of trace homomorphism on rings of integers}
\author[F. Battistoni]{Francesco Battistoni}

\begin{document}

\maketitle
\begin{abstract}
In this note we give a brief survey of the most elementary criteria used to determine the surjectivity of the trace operator on the ring of integers of a number field $K$. Furthermore, we introduce an easy to state yet unknown surjectivity criterion depending only on the prime factorization of the degree $n$ of $K$ and on the squarefree part of the discriminant $d_K$.
\end{abstract}

\section{Preliminaries and trace homomorphism}
Let $K$ be a number field of degree $n\in\N$ over the field $\Q$ of rational numbers. It is known that, for $n> 1$, there is not a canonical way to embed $K$ in the field $\C$ of complex numbers; nonetheless, the field $K$ admits exactly $n$ embeddings $\sigma_1,\ldots,\sigma_n: K\rightarrow\C$.\\
By the Primitive Element Theorem (Theorem 5.1 of \cite{milneFT}) we know that any number field $K$ has the form $\Q(\alpha)$ for some algebraic number $\alpha\in K$, and if $p(x)\in\Z[x]$ is the minimum polynomial of $\alpha$, then there is a bijection between the embeddings $\sigma_1,\ldots,\sigma_n$ of $K$ and the complex roots of $p(x)$. 
\\\\
Given $\beta\in K$, define its \textbf{trace} as the number $\Tr(\beta):= \sum_{i=1}^n \sigma_i(\beta).$ By its very definition, the trace is an algebraic number which is invariant for the action of the $n$ embeddings of $K$, and thus it is a rational number. This allows to define the trace function
$$\Tr: K\rightarrow\Q$$
which is immediately seen to be an homomorphism of $\Q$-vector spaces.
\\\\
If $K=\Q(\alpha)$ and $p(x):= x^n +a_1 x^{n-1} + \cdots +a_{n-1}x +a_n\in\Q[x]$ is the minimum polynomial of $\alpha$, then $\Tr(\alpha) = -a_1$; this follows immediately from the fact that $p(x)$ splits as $\prod_{i=1}^n (x-\sigma_i(\alpha))$ in any algebraic closure of $K$ (Corollary 3.12 of \cite{jarvisAlgebraic}).\\\\
Let $\Ok$ be the ring of integers of $K$, i.e. the subring of the algebraic integers contained in $K$. If $\alpha\in\Ok$, then not only $\Tr(\alpha)$ is a rational number but it is also an algebraic integer, and so $\Tr(\alpha)\in\Z$. The restricted map 
$$\Tr: \Ok\rightarrow \Z$$
is an homomorphism of abelian groups.
\\\\
The ring of integers $\Ok$ satisfies the following two important properties:
\begin{itemize}
    \item Any non-zero ideal $I\subset\Ok$ can be written in a unique way as a finite product of prime ideals of $\Ok$ (Theorem 3.14, Chapter I of \cite{janusz});
    \item If $K$ has degree $n$, then $(\Ok, +)$ is a free abelian group of rank $n$ (Theorem 1, Chapter I of \cite{lang2013algebraic}).
\end{itemize}
\section{Surjectivity of trace operator}\label{SectionSurjectivity}
Given a number field $K$ of degree $n$, it is very easy to see that the trace map is a surjective homomorphism: in fact, $\Tr(1) = n$ and so, given $a/b\in\Q$, the element $a/(nb)$ is such that $\Tr(a/(nb)) = a/(nb)\cdot\Tr(1) = a/b$.\\
Actually, this proves that considering the subfield $\Q\subset K$ is enough to yield a surjection.
\\\\
Does this surjectivity hold also for the restricted map $\Tr:\Ok\rightarrow\Z$? Surely the trick of dividing by the degree $n$ of the number field no longer works, because given $\alpha\in\Ok$ the element $\alpha/n$ may not be in $\Ok$.\\
In fact, it is very easy to provide an example of number field for which the trace restricted to the ring of integers is not surjective: consider the field $K=\Q(\sqrt{2})$, which has minimum polynomial $p(x):= x^2-2$. The ring of integers $\Ok$ is then equal to $\Z[\sqrt{2}]$ (Propostions 1.32 and 1.33, Chapter II of \cite{frohlichtaylor}), i.e. any algebraic integer in $K$ has the form $a+b\sqrt{2}$ with $a,b\in\Z$. Being $\Tr(m)=2m$ for any $m\in\Z$ and $\Tr(\sqrt{2})=0$ because of $p(x)$, then the trace of any element of $\Ok$ is an even rational integer, and so the restricted trace map is not surjective.\\
The above considerations imply that the restricted trace is not surjective for any quadratic number field $\Q(\sqrt{d})$ with $d\in\Z$ squarefree and $d\equiv 2,3$ mod $4$ (this last assumption is needed to ensure that the ring of integers is equal to $\Z[\sqrt{d}]$).
\\\\
One could wonder if there exist any criteria, different from explicitly studying the trace map, to determine whether the restricted trace homomorphism is surjective.\\
A first try comes from looking at the minimum polynomial of the number field.
\begin{prop}
Let $K$ be a number field with minimum polynomial $p(x):= x^n +a_1 x^{n-1} + \cdots +a_{n-1}x +a_n\in\Z[x].$ If $a_1=\pm 1$, then the trace map $\Tr:\Ok\rightarrow\Z$ is surjective.
\end{prop}
\begin{proof}
$p(x)$ being a monic irreducible polynomial with integer coefficients, there exists $\alpha\in\Ok$ root of $p(x)$ such that $\Tr(\alpha)= -a_1=\mp 1$. Then, for every $m\in\Z$ it is $m = \Tr(m\alpha) $ or $m=\Tr(-m\alpha)$ depending on the sign of $\Tr(\alpha)$.
\end{proof}
\noindent
What can be said for number fields $K$ which are defined by polynomials with coefficient $a_1\neq \pm 1$ and such that it seems not possible to produce elements $\alpha\in \Ok$ with $\Tr(\alpha)=\pm 1$ by hands only? One can get further information thanks to the concept of ramification, which is naturally related to the trace homomorphism: this is the subject of the next section.

\section{Discriminants and ramification}
Let $K$ be a number field and let $\Ok$ be its ring of integers. 
Given a prime number $p\in\Z$, the ideal $p\Ok$ is not necessarily prime but has a factorization
$$p\Ok = \pp_1^{e_1}\cdots\pp_r^{e_r}$$
where the $\pp_i$'s are prime ideals in $\Ok$ and $e_i\in\N$ for every $i=1,\ldots,r$. The prime number $p$ is said to be \textbf{ramified in $K$} if $e_i>1$ for some index $i$.\\\\
It is a classical problem in Number Theory to detect the prime numbers ramifying in a number field $K$: its solution depends mainly on the folllowing concepts.\\
Let $\alpha_1\ldots,\alpha_n\in\Ok$ be independent $\Z$-generators of $\Ok$ as abelian group. The \textbf{discriminant of $K$} is defined as
$$d_K:= ( \det (\sigma_i(\alpha_j))_{i,j=1}^n )^2 = \det (\Tr(\alpha_i\alpha_j))_{i,j=1}^n.$$
One gets $d_K\in\Z$ because of the last equality, and it is obvious from the definition that the value of $d_K$ does not change by considering a new system $\beta_1,\ldots,\beta_n$ of $\Z$-independent generators for $\Ok$.\\ The importance of the discriminant for the study of the ramified primes relies in the following proposition:
\begin{prop}
A prime number $p$ ramifies in $K$ if and only if $p$ divides $d_K$.
\end{prop}
\begin{proof}
See Corollary III.2.12 of \cite{neukirch}.
\end{proof}
The prime numbers may ramify with different behaviours: the following distinction will be useful to provide criteria for the study of the restricted trace homomorphism.
\\
Let $p$ be a rational prime number ramifying in $K$ and let $p\Ok = \pp_1^{e_1}\cdots \pp_r^{e_r}$ be its prime ideal factorization in $\Ok$. Then $p$ is said to be \textbf{wildly ramified} if there exists $i\in\{1,\ldots,r\}$ such that $p$ divides $e_i$; otherwise $p$ is said to be \textbf{tamely ramified}.\\
A number field $K$ is said to be \textbf{tame} if every ramified prime number is tamely ramified, otherwise $K$ is said to be \textbf{wild}.
\\\\
The last tool needed is the concept of different ideal.\\
Consider the set $\Hat{\Ok}:=\{\alpha\in K\colon \Tr(\alpha\cdot\Ok)\subset\Z\}$. The set $\Dk:=\{\beta\in K\colon \beta\cdot\Hat{\Ok}\subset\Ok\}$ is called the \textbf{different ideal of $K$} (or simply the different of $K$); it is an abelian group with respect to the sum.
\begin{lem}
The different $\Dk$  satifies the following properties:
\begin{itemize}
    \item $\Dk$ is an ideal of $\Ok$;
    \item If $p$ wildly ramifies in $K$, $p\Ok = \pp_1^{e_1}\cdots\pp_r^{e_r}$ and there exists $i\in\{1,\ldots,r\}$ such that $p$ divides $e_i$, then $\pp_i$ is a factor of $\Dk$ with exponent at least $e_i$;
    \item If $p$ tamely ramifies in $K$ and  $p\Ok = \pp_1^{e_1}\cdots\pp_r^{e_r}$, then for any $i\in\{1,\ldots,r\}$ the number $e_i-1$ is the exact exponent of the prime $\pp_i$ as factor of $\Dk$;
    \item The size of the quotient ring $\Ok/\Dk$ is equal to $|d_K|$.
    \end{itemize}
\end{lem}
\begin{proof}
These results are all proved in Section 4.2 of \cite{narkiewicz2013elementary}.
\end{proof}
The distinction between tame and wild number fields and the concept of different ideal have proved to be important in determining the surjectivity of the trace homomorphism restricted to the ring of integers. 
\begin{thm}\label{Tame}
Let $K$ be a tame number field. Then $\Tr:\Ok\rightarrow\Z$ is surjective.
\end{thm}

\begin{proof}
See Corollary 5, Section 4.2 of \cite{narkiewicz2013elementary}. The different ideal has a main role in the setting of the proof.
\end{proof}

One can get an interesting Corollary, from which the surjectivity of the restricted trace can be recovered by looking only at the factorization of the discriminant.

\begin{coro}\label{DiscCriterion}
Let $K$ be a number field with squarefree discriminant $d_K$. Then $\Tr:\Ok\rightarrow\Z$ is surjective.
\end{coro}

\begin{proof}
If $d_K=\pm p_1\cdots p_r$ is squarefree, then $\Dk = \pp_1\cdots\pp_r$ where the size of every quotient ring $\Ok/\pp_i$ is equal to $p_i$. This implies that, for any fixed factor $p_i$ of the discriminant, $\pp_i$ is the unique factor of $p_i\Ok$ which has exponent greater than 1, and the value of this exponent is precisely equal to 2.
Thus, any odd $p_i$ is tamely ramified.\\
If $2$ divides $d_K$ and $\Q$ is the factor of $2\Ok$ dividing $\Dk$, then either 2 wildly ramifies with the exponent of $\Q$ being 1, or 2 tamely ramifies with the exponent of $\Q$ being 2, and both these options are absurd.\\
Thus $K$ is a tame number field, and from Theorem \ref{Tame} the surjectivity on the trace over the ring of integers follows.
\end{proof}

\section{A weaker discriminant criterion}
Theorem \ref{Tame} of the previous section proves the surjectivity of the restricted trace for a wide class of number fields, and it also yields a good sufficient criterion depending only on the factorization of the discriminant $d_K$.\\
The goal of this section is to present a simple, yet new, criterion for the surjectivity which not only relies on the factorization of $d_K$, but has the advantage to give a positive answer also for some wild number fields.

\begin{thm}\label{CriterioNuovo}
Let $K$ be a number field of degree $n$ and assume that, for every prime number $p$ dividing $n$, the number $p^2$ does not divide $d_K$. Then $\Tr:\Ok\rightarrow \Z$ is surjective.
\end{thm}

\begin{proof}
Let $T_0(K):=\{\alpha\in\Ok\colon \Tr(\Ok)=0\}$ be the kernel of the restricted trace homomorphism. The structure theorem of free abelian groups (Theorem 7.3, Chapter I of \cite{langAlgebra}) implies that $T_0(K)$ is a free abelian group too, its rank being equal to $n-1$.\\
The set $\Tr(\Ok)$ is an ideal in $\Z$. Let $t$ be the positive generator of this ideal. Since
$n = \Tr(1)$ one gets that $t$ divides $n$.

The previous considerations imply that the ring of integers admits a decomposition $\Ok = T_0(K)\oplus \Z\gamma$ as free abelian group, where $\gamma\in\Ok$ is such that $\Tr(\gamma)=t$. Let $\alpha_1,\ldots,\alpha_{n-1}$ be a $\Z$-basis for $T_0(K)$: then $\alpha_1,\ldots,\alpha_{n-1},\gamma$ is a $\Z$-basis for $\Ok$ and so the discriminant $d_K$ can be computed by means of this basis.\\
Let $M_K$ denote the matrix
$$ 
\left(
\begin{matrix}
 \sigma_1(\alpha_1) & \cdots & \sigma_n(\alpha_1)\\
 \cdots &\cdots &\cdots\\
 \sigma_1(\alpha_{n-1}) &\cdots &\sigma_n(\alpha_{n-1})\\
 \sigma_1(\gamma) &\cdots &\sigma_n(\gamma)
\end{matrix}
\right).$$
Since its determinant does not change by replacing the last column with the
sum of every other column, we get that
$$\det M_K = \det \left(
\begin{matrix}
 \sigma_1(\alpha_1) & \cdots & \sigma_{n-1}(\alpha_1) & \Tr(\alpha_1)\\
 \cdots &\cdots &\cdots&\cdots\\
 \sigma_1(\alpha_{n-1}) &\cdots &\sigma_{n-1}(\alpha_{n-1})&\Tr(\alpha_{n-1})\\
 \sigma_1(\gamma) &\cdots &\sigma_{n-1}(\gamma)&\Tr(\gamma)
\end{matrix}
\right)
$$
$$
=
\det \left(
\begin{matrix}
 \sigma_1(\alpha_1) & \cdots & \sigma_{n-1}(\alpha_1) & 0\\
 \cdots &\cdots &\cdots&\cdots\\
 \sigma_1(\alpha_{n-1}) &\cdots &\sigma_{n-1}(\alpha_{n-1})& 0\\
 \sigma_1(\gamma) &\cdots &\sigma_{n-1}(\gamma)&t
\end{matrix}
\right).
$$
Consider now the minor given by the first $n-1$ rows and the first $n-1$ columns: 
$$
N_K:=\left(
\begin{matrix}
 \sigma_1(\alpha_1) & \cdots & \sigma_{n-1}(\alpha_1) \\
 \cdots &\cdots &\cdots\\
 \sigma_1(\alpha_{n-1}) &\cdots &\sigma_{n-1}(\alpha_{n-1})
\end{matrix}
\right).
$$
Applying any $\sigma_i$ which is not the identity embedding on $K$, one sees that a column of $N_K$ is now formed by elements $\sigma_n(\alpha_j)$ with $j=1,\ldots,n-1$, while the other columns are permutations of the remaining columns. But for every $j\in\{1,\ldots,n-1\}$ it is $\sigma_n(\alpha_j) = -\sum_{i=1}^{n-1}\sigma_i(\alpha_j)$: this implies that $\det N_K$ is invariant for the action of the embeddings $\sigma_i$, up to a possible change of sign due to the permutation of the columns.
\\
Thus it is enough to take the square of $\det N_K$ to get an algebraic integer invariant for any embeddings $\sigma_i$, i.e. a rational integer, and so
$d_K = (\det M_K)^2 = (\det N_K)^2 \cdot t^2 = C\cdot t^2$, with $C\in\Z$. In other words, it is $d_K/t^2\in\Z$.\\
Finally, from the above considerations and the fact that $t$ divides $n$, the hypothesis of the theorem force $t=1$, and so the trace $\Tr:\Ok\rightarrow\Z$ must be surjective.
\end{proof}

An example of wild number field for which the surjectivity of the restricted trace is not evident without Theorem \ref{CriterioNuovo} is given by the cubic field $K$ defined by the polynomial $x^3+x-6$. In fact, its discriminant is equal to $-2^2\cdot 61$, so the primes 61 and 2
both ramify. A computation with the computer algebra package PARI/GP \cite{pari} shows
that the prime 2 wildly ramifies, thus the extension is not tame and Theorem \ref{Tame} or
Corollary \ref{DiscCriterion} do not apply. However, $\Tr : \Ok \to \Z$ is surjective, by Theorem \ref{CriterioNuovo}.\\\\
Some final considerations arise looking back at the quadratic fields studied in Section \ref{SectionSurjectivity}: in fact, these are wild fields for which the trace on the ring of integers is not surjective. Moreover, their discriminant is always divided by $4=2^2$, and thus they do not satisfy the hypotheses needed for the sufficient criterion introduced by Theorem \ref{CriterioNuovo}.\\
This suggests a possible conjecture for the complete characterization of the surjectivity of the trace map on the ring of integers: \\
\textit{Given a number field $K$, then $\Tr:\Ok\rightarrow\Z$ is not surjective if and only if $K$ is wild and does not satisfy the hypotheses of Theorem \ref{CriterioNuovo}.}

\end{document}